\documentclass[12pt]{article}
\usepackage{amssymb,amsmath}
\usepackage{amsthm}
\usepackage{cases}
\usepackage{amsfonts}
\usepackage{cite,color,xcolor}
\usepackage[left=2.4cm,right=2.4cm,top=2.1cm,bottom=2.1cm]{geometry}
\usepackage[colorlinks,citecolor=blue,urlcolor=blue]{hyperref}
\usepackage[utf8]{inputenc}
 \usepackage[pagewise]{lineno}

\newtheorem{theorem}{Theorem}[section]

\newtheorem{lemma}{Lemma}[section]
\newtheorem{proposition}{Proposition}[section]

\newtheorem{definition}{Definition}[section]
\newtheorem{remark}{Remark}[section]

\usepackage{txfonts}

\newcommand{\bal}{\begin{align}}
\newcommand{\bbal}{\begin{align*}}
\newcommand{\beq}{\begin{equation}}
\newcommand{\eeq}{\end{equation}}
\newcommand{\bca}{\begin{cases}}
\newcommand{\eca}{\end{cases}}

\newcommand{\pa}{\partial}
\newcommand{\fr}{\frac}
\newcommand{\na}{\nabla}
\newcommand{\De}{\Delta}

\newcommand{\ep}{\varepsilon}
\newcommand{\dd}{\mathrm{d}}
\newcommand{\B}{\dot{B}}

\newcommand{\LL}{\tilde{L}}
\newcommand{\R}{\mathbb{R}}

\newcommand{\D}{\mathrm{div}}

\newcommand{\ee}{\vec{e}}

\begin{document}
\title{Ill-posedness issue on a multidimensional chemotaxis equations in the critical Besov spaces}

\author{Jinlu Li$^{1}$, Yanghai Yu$^{2,}$\footnote{E-mail: lijinlu@gnnu.edu.cn; yuyanghai214@sina.com(Corresponding author); mathzwp2010@163.com} and Weipeng Zhu$^{3}$\\
\small $^1$ School of Mathematics and Computer Sciences, Gannan Normal University, Ganzhou 341000, China\\
\small $^2$ School of Mathematics and Statistics, Anhui Normal University, Wuhu 241002, China\\
\small $^3$ School of Mathematics and Big Data, Foshan University, Foshan, Guangdong 528000, China}

\date{\today}

\maketitle\noindent{\hrulefill}

{\bf Abstract:} In this paper, we aim to solving the open question left in [Nie, Yuan: Nonlinear Anal 196 (2020); J. Math. Anal. Appl 505 (2022) and Xiao, Fei: J. Math. Anal. Appl 514 (2022)]. We prove that a multidimensional chemotaxis system is ill-posedness in $\dot{B}_{2d, r}^{-\frac{3}{2}} \times\big(\dot{B}_{2d, r}^{-\frac{1}{2}}\big)^{d}$ when $1\leq r<d$ due to the lack of continuity of the solution.


\vskip0mm\noindent{\hrulefill}

\section{Introduction}

In this paper, we consider the Cauchy problem of the following multidimensional ($d\geq2$) chemotaxis equations
\begin{equation}\label{che}
\begin{cases}
\pa_tu-\De u=\D(uv),& (t,x)\in\R^+\times\R^d,\\
\pa_tv-\na u=0,&(t,x)\in\R^+\times\R^d,\\
(u,v)(0,x)=(u_0(x),v_0(x)),& x\in\R^d,
\end{cases}
\end{equation}
here the scalar unknown function $u(t, x)$ represents the cell density and the vector unknown function $v(t, x)=-\nabla\ln c$, where $c$ is the chemical concentration.

For more than a century, biologists have observed that certain species of bacteria are preferred to move
toward higher concentrations of some chemicals, such as minerals, oxygen and organic nutrients. This biased movement, universally referred to as chemotaxis, has been fueling interest of both experimentalists and
theoreticians since it plays a vital role in wide-ranging biology phenomena \cite{hil}.
Many diverse disciplines involve chemotaxis models whose
aspects include not only the mechanistic basis and biological foundations but also the modeling of specific systems and the mathematical analysis of the governing nonlinear equations. The following classical Keller-Segel type chemotaxis model \cite{Ke1} reads:
\begin{equation}\label{lim}
\begin{cases}
\partial_{t} u=\D\left(\mu \nabla u-\chi u \nabla \Phi(c)\right), \\
\tau \partial_{t} c=\kappa \Delta c+g(u, c),
\end{cases}
\end{equation}
where $u$ stands for the cell density and $c$ for the chemical concentration. We denote by $\mu>0$ the diffusion rate of the cells and $\kappa \geq 0$ the diffusion rate of the chemical substance, and denote $\chi>0$ and $\chi<0$ as the attractive chemotaxis and the repulsive chemotaxis, respectively. The relaxation time scale $\tau$ is non-negative, and the function $\Phi(c)$ is the chemotactic potential function and $g(u, c)$ is the chemical kinetics.
As mentioned in \cite{lpz}, the Keller-Segel model
of chemotaxis \cite{Ke1,Ke2,Ke3} has provided that, a cornerstone for much of these works,
its success being a consequence of its intuitive simplicity, analytical tractability, and capability to model
the basic dynamics of chemotactic populations.

For the Keller-Segel model \eqref{lim}, there are two limiting cases: $\kappa\to0$ and $\tau\to0$. In this present paper, we only
consider the case $\kappa\to0$. We take $\kappa=0,\mu=\tau=1$, $\Phi(c)=\ln c$, $g(u, c)=-u c$
and $\chi > 0$ which corresponds to the attractive
chemotaxis, then \eqref{lim} reduces
\begin{equation}\label{lim1}
\begin{cases}
\partial_{t} u=\Delta u-\D\left(\chi u \nabla \ln c\right), \\
\partial_{t} c=-uc.
\end{cases}
\end{equation}
Setting  $v=-\nabla \ln c$, by suitable scaling, \eqref{lim1} becomes the hyperbolic-parabolic model \eqref{che}. We can check that if $(u, v)$ solves \eqref{che}, so does $(u_\ell,v_\ell)$ where
$(u_\ell,v_\ell)(t,x)=(\ell^{2} u(\ell^{2} t, \ell x), \ell v(\ell^{2} t, \ell x)).$ This suggests us to choose initial data $(u_0,v_0)$ in ``critical spaces" whose norm is invariant (up to a constant independent of $\ell$) for all $\ell>0$ by the transformation $(u_{0}, v_{0})(x) \mapsto(\ell^{2} u_{0}(\ell x), \ell v_{0}(\ell x))$.
It is natural that $\dot{H}^{\frac{d}{2}-2}(\R^d) \times(\dot{H}^{\frac{d}{2}-1}(\R^d))^{d}$ and $\dot{B}_{p, q}^{\frac{d}{p}-2}(\R^d) \times(\dot{B}_{p, q}^{\frac{d}{p}-1}(\R^d))^{d}$are critical spaces to  \eqref{che}.

There is a huge literature on the studies of the well-posedness problem and long-time behaviors of solutions for the Keller-Segel type
chemotaxis model due to its capturing the principal features of the basic
dynamics of chemotactic population.  For more background of the chemotaxis model and more relevant results, we refer to \cite{cx,tao1,tao3,tao4,win1,win2,win3,win4,win5,win6,win7,win8,win9,zzc,Zhang1,Zhang2,Zhang3,Zhang4} and references therein. Next, concerned with the model \eqref{che}, we briefly review some results which focused on well-posedness problems. Li-Wang \cite{lw1,lw2} proved nonlinear stability of traveling waves of
arbitrary amplitudes to repulsive chemotaxis model. Li-Li-Zhao \cite{llz} established the local and
global well-posedness in the Sobolev space $H^s$ with $s > 1+d/2$. They also showed that solution converges exponentially to the constant steady
state with a frequency-dependent decay rate as time goes to infinity when the initial data is suitably close
to a constant positive steady state. In the case of one dimension, Li-Pan-Zhao \cite{lpz} proved the global
existence of classical solutions and the solutions converge exponentially to constant equilibrium states in time
for large initial data. Moreover, they obtained similar results for the multidimensional model when the
initial data are small. Hao \cite{hao} showed the global existence and uniqueness of the strong solution for initial data close to a constant equilibrium state in critical Besov spaces $\dot{B}_{2, 1}^{\frac{d}{2}-2}(\R^d) \times(\dot{B}_{2, 1}^{\frac{d}{2}-1}(\R^d))^{d}$. Nie-Yuan \cite{n1} established the local well-posedness and global well-posedness of \eqref{che} with small initial data in $\dot{B}_{p, 1}^{\frac{d}{p}-2}(\R^d) \times(\dot{B}_{p, 1}^{\frac{d}{p}-1}(\R^d))^{d}$ when $1\leq p<2 d$. For the related Keller-Segel systems, many results involving the  finite-time blow-up of solution are available, we refer to \cite{KS,La,win3} and references therein.

In this paper, we are mainly focused on ill-posedness of solutions to system \eqref{che} in some critical Besov spaces. From the PDE's point of view, it is crucial to know if an equation which models a physical phenomenon is well-posed in the
Hadamard's sense: existence, uniqueness, and continuous dependence of the solutions with respect to the initial data. In particular, the lack of continuous dependence would cause incorrect solutions or non meaningful solutions. Indeed, this means that the corresponding equation is ill-posed. Many results with regard to the ill-posedness have been obtained for some important nonlinear PDEs including the incompressible Navier-Stokes equations \cite{Bou,wang,yon}, the stationary Navier-Stokes equations \cite{Tsu1,Lyz}, the compressible Navier-Stokes equations \cite{Chen,Iwa} and so on. Recently, there have been a few results about ill-posedness of
system \eqref{che} in critical Besov spaces.
Nie-Yuan \cite{n1} proved that \eqref{che} is ill-posed in $\dot{B}_{p, 1}^{\frac{d}{p}-2}(\R^d) \times(\dot{B}_{p, 1}^{\frac{d}{p}-1}(\R^d))^{d}$ when $p>2 d$. Later on, for the critical case $p=2 d$, Nie-Yuan \cite{n2} further proved that \eqref{che} is ill-posed in $\dot{B}_{2d, 1}^{-\frac{3}{2}}(\R^d) \times(\dot{B}_{2d, 1}^{-\frac{1}{2}}(\R^d))^{d}$ by fully exploiting nonlinear structure of the cross term. Subsequently, Xiao-Fei \cite{x} proved the ill-posedness of \eqref{che} in $\dot{B}_{2 d, r}^{-\frac{3}{2}}(\R^d) \times(\dot{B}_{2 d, r}^{-\frac{1}{2}}(\R^d))^{d}$ for $r>2$ by a different framework which give a special initial data. Obviously, there are still gaps on the index $r$ between Nie-Yuan and Xiao-Fei's ill-posedness results.  Precisely speaking, for the case $1<r\leq 2$, it is still unknown whether \eqref{che} in $\dot{B}_{2 d, r}^{-\frac{3}{2}}(\R^d) \times(\dot{B}_{2 d, r}^{-\frac{1}{2}}(\R^d))^{d}$ is well-posed or ill-posed. In this paper, we shall answer this question.

\subsection{Main Result}
\quad
The main result of this paper is the following:
\begin{theorem}\label{th3}
Let $d\geq 2$ and $1\leq r<d$. \eqref{che} is ill-posed in $\dot{B}_{2 d, r}^{-\frac{3}{2}}(\R^d) \times\Big(\dot{B}_{2 d, r}^{-\frac{1}{2}}(\R^d)\Big)^{d}$  in the following sense: There exists a sequence of initial data $ \Big\{(u_{0,n},v_{0,n})\Big\}_{n=1}^{\infty}\subset \dot{B}_{2 d, r}^{-\frac{3}{2}}\times\Big(\dot{B}_{2 d, r}^{-\frac{1}{2}}\Big)^{d}$ satisfying
$$
\lim _{n \rightarrow \infty}\left(\|u_{0,n}\|_{\B^{-\frac32}_{2d,r}} +\|v_{0,n}\|_{\B^{-\frac12}_{2d,r}}\right)= 0,
$$
 such that the corresponding solution $(u,v)$ to \eqref{che} satisfies
$$
\|u(t_n,\cdot)\|_{\B^{-\frac32}_{2d,r}} +\|v(t_n,\cdot)\|_{\B^{-\frac12}_{2d,r}} \geq c(\ep)>0 \quad  with \quad t_n= \ep2^{-2n},
$$
where $\ep$ is some sufficiently small positive constant.
\end{theorem}
\begin{remark}
Theorem \ref{th3} demonstrates that if $d\geq 2$ and $1\leq r<d$, there exists a sequence of initial data which converges to zero in $\dot{B}_{2 d, r}^{-\frac{3}{2}}(\R^d) \times\Big(\dot{B}_{2 d, r}^{-\frac{1}{2}}(\R^d)\Big)^{d}$ and yields a sequence of solutions to \eqref{che} which does not converge to zero in $\dot{B}_{2 d, r}^{-\frac{3}{2}}(\R^d) \times\Big(\dot{B}_{2 d, r}^{-\frac{1}{2}}(\R^d)\Big)^{d}$. In other words, \eqref{che} is ill-posed in $\dot{B}_{2 d, r}^{-\frac{3}{2}}(\R^d) \times\Big(\dot{B}_{2 d, r}^{-\frac{1}{2}}(\R^d)\Big)^{d}$ due to the discontinuity of the solution map at zero.
\end{remark}
\begin{remark}
Theorem \ref{th3} enriches the ill-posedness theories of the system \eqref{che} although our discontinuity of the solution map in Theorem \ref{th3} is weaker than ``Norm Inflation" of \cite{n1,n2,x}.
\end{remark}
\begin{remark}
We should emphasize that the only question left is the well/ill-posedness of the system \eqref{che} in $\dot{B}_{4, 2}^{-\frac{3}{2}}(\R^2) \times\Big(\dot{B}_{4, 2}^{-\frac{1}{2}}(\R^2)\Big)^{2}$.
\end{remark}

\subsection{Main Idea}
\quad
From $\eqref{che}_2$, one has
\bal\label{v0}
v(t,x)&=v_0+\int_0^t\nabla u\dd\tau.
\end{align}
By the Duhamel formula, we obtain from $\eqref{che}_1$ that
\begin{align}
u(t, x)&=e^{t \Delta} u_{0}+\int_{0}^{t} e^{(t-s) \Delta}\D \left(uv \right)\mathrm{d} s\nonumber\\
&=\underbrace{e^{t \Delta} u_{0}}_{=:U_1}+\underbrace{\int_{0}^{t} e^{(t-s) \Delta}\D \left\{U_1\left(v_{0}+ \int_{0}^{s} \nabla U_1 \dd \tau\right) \right\}\mathrm{d}s}_{=:U_2}+\underbrace{\text{Remainder term}}_{=:U_3}.
\end{align}
Then, we decompose $u$ into three terms, namely, $u=U_1+U_2+U_3$. Thus
\bal
v(t,x)
&=\underbrace{v_0+\int_0^t\nabla U_1\dd\tau}_{=:V_1}+\underbrace{\int_0^t\nabla U_2\dd\tau}_{=:V_2}+\underbrace{\int_0^t\nabla U_3\dd\tau}_{=:V_3}.
\end{align}
For the convenience of using regularity estimate of heat equations later, we deduce that $U_i$ ($i=1,2,3$) solves the following three equations respectively,
\begin{equation}\label{1}
\begin{cases}
\pa_tU_1-\De U_1=0,\\
U_1|_{t=0}=u_0,
\end{cases}
\end{equation}
\begin{equation}\label{2}
\begin{cases}
\pa_tU_2-\De U_2=\D(U_1V_1),\\
U_2|_{t=0}=0,
\end{cases}
\end{equation}
and
\begin{equation}\label{3}
\begin{cases}
\pa_tU_3-\De U_3=\D\mathbf{F},\\
U_3|_{t=0}=0,
\end{cases}\end{equation}
where
\bbal
\mathbf{F}:=U_3V_3+U_3(V_2+V_1)+V_3(U_1+U_2)+U_1V_2+U_2(V_1+V_2).
\end{align*}

To make sure that $U_2$ leads to the discontinuity, we decompose it as
$$U_2=\underbrace{\int_{0}^{t} e^{(t-s) \Delta}\D (u_0v_{0})\mathrm{d}s}_{=:U_{2,1}}+\underbrace{\int_{0}^{t} e^{(t-s) \Delta}\D \left((U _1-u_0)v_{0}+U_1\int_{0}^{s} \nabla U_1 \dd \tau \right)\mathrm{d}s}_{=:U_{2,2}}.$$
We try to extract the worst term $U_{2,1}$. Our key argument is that, by constructing suitable initial data $(u_0,v_0)$, the other terms in $U_{2}$ can be absorbed by $U_{2,1}$ and the terms $U_1, U_3$ can be small. Precisely speaking, the term $U_{2,1}$ is the main contribution
to the discontinuity.
\section{Littlewood-Paley analysis}
\quad
Next, we will recall some facts about the Littlewood-Paley decomposition, the homogeneous Besov spaces and their some useful properties.

Choose a radial, non-negative, smooth function $\chi:\R^d\mapsto [0,1]$ such that it is supported in $\mathcal{B}:=\{\xi\in\mathbb{R}^d:|\xi|\leq 4/3\}$  and $\chi\equiv1$ for $|\xi|\leq3/4$. Setting $\varphi(\xi):=\chi(\xi/2)-\chi(\xi)$, then we deduce that $\varphi$ is supported in $\mathcal{C}:=\{\xi\in\mathbb{R}^d: 3/4\leq|\xi|\leq 8/3\}$. In particular, it holds that $\varphi(\xi)\equiv 1$ for $4/3\leq |\xi|\leq 3/2$ which will be used in the sequel.
For every $u\in \mathcal{S'}(\mathbb{R}^d)$, the homogeneous dyadic blocks ${\dot{\Delta}}_j$ is defined as follows
\bbal
&\dot{\Delta}_ju=\varphi(2^{-j}D)u=\mathcal{F}^{-1}\big(\varphi(2^{-j}\cdot)\mathcal{F}u\big)
=2^{dj}\int_{\R^d}\check{\varphi}\big(2^{j}(x-y)\big)u(y)\dd y,\quad \forall j\in \mathbb{Z}.
\end{align*}
In the homogeneous case, the following Littlewood-Paley decomposition makes sense
$$
u=\sum_{j\in \mathbb{Z}}\dot{\Delta}_ju\quad \text{for any}\;u\in \mathcal{S}'_h(\R^d),
$$
where $\mathcal{S}'_h$ is given by
\begin{eqnarray*}
\mathcal{S}'_h:=\Big\{u \in \mathcal{S'}(\mathbb{R}^{d}):\; \lim_{j\rightarrow-\infty}\|\chi(2^{-j}D)u\|_{L^{\infty}}=0 \Big\}.
\end{eqnarray*}
We turn to the definition of the Besov Spaces and norms which will come into play in our paper.
\begin{definition}[see \cite{B}]
Let $s\in\mathbb{R}$ and $(p,r)\in[1, \infty]^2$. The homogeneous Besov space $\dot{B}^s_{p,r}(\R^d)$ consists of all tempered distribution $f$ such that
\begin{align*}
\dot{B}_{p,r}^{s}=\Big\{f \in \mathcal{S}'_h(\mathbb{R}^{d}):\; \|f\|_{\dot{B}_{p,r}^{s}(\mathbb{R}^{d})}< \infty \Big\},
\end{align*}
where
\begin{numcases}{\|f\|_{\dot{B}^{s}_{p,r}(\R^d)}:=}
\left(\sum_{j\in\mathbb{Z}}2^{sjr}\|\dot{\Delta}_jf\|^r_{L^p(\R^d)}\right)^{1/r}, &if $1\leq r<\infty$,\nonumber\\
\sup_{j\in\mathbb{Z}}2^{sj}\|\dot{\Delta}_jf\|_{L^p(\R^d)}, &if $r=\infty$.\nonumber
\end{numcases}
\end{definition}

For $0<T \leq \infty, s \in \mathbb{R}$ and $1 \leq p, r, \rho \leq \infty$, we set (with the usual convention if $r=\infty$ )
$$
\|f\|_{\tilde{L}_{T}^{\rho}\left(\dot{B}_{p, r}^{s}\right)}:=\left(\sum_{j \in \mathbb{Z}} 2^{j s r}\left\|\dot{\Delta}_{j} f\right\|_{L^{\rho}\left(0, T ; L^{p}\right)}^{r}\right)^{1/r}.
$$

The following Bernstein's inequalities will be used in the sequel.
\begin{lemma}[see \cite{B}] \label{lem2.1} Let $\mathcal{B}$ be a ball and $\mathcal{C}$ be an annulus. There exists a constant $C>0$ such that for all $k\in \mathbb{N}\cup \{0\}$, any positive real number $\lambda$ and any function $f\in L^p$ with $1\leq p \leq q \leq \infty$, we have
\begin{align*}
&{\rm{supp}}\widehat{f}\subset \lambda \mathcal{B}\;\Rightarrow\; \|\nabla^kf\|_{L^q}\leq C^{k+1}\lambda^{k+(\frac{d}{p}-\frac{d}{q})}\|f\|_{L^p},  \\
&{\rm{supp}}\widehat{f}\subset \lambda \mathcal{C}\;\Rightarrow\; C^{-k-1}\lambda^k\|f\|_{L^p} \leq \|\nabla^kf\|_{L^p} \leq C^{k+1}\lambda^k\|f\|_{L^p}.
\end{align*}
\end{lemma}
As a direct result of Bernstein's inequalities, we have the following continuous embedding:
\begin{lemma}[see \cite{B}]
Let $s\in\R$, $1 \leq p_{1} \leq p_{2} \leq \infty$ and $1 \leq r_{1} \leq r_{2} \leq \infty$. Then $$\dot{B}_{p_{1}, r_{1}}^{s}(\R^d)\hookrightarrow\dot{B}_{p_{2}, r_{2}}^{t}(\R^d)\quad\text{with}\quad t=s-\left(\frac{d}{p_{1}}-\frac{d}{p_{2}}\right).$$
\end{lemma}
\begin{lemma}[see \cite{B}]\label{p}
Let $s>0,1 \leq p \leq \infty$ and $1 \leq \rho, \rho_{1}, \rho_{2}, \rho_{3}, \rho_{4} \leq \infty$. Then
$$
\|fg\|_{\LL_{T}^{\rho}\left(\dot{B}_{p, r}^{s}\right)} \leq C\left(\|f\|_{L_{T}^{\rho_{1}}\left(L^{\infty}\right)}\|g\|_{\LL_{T}^{\rho_{2}}\left(\dot{B}_{p, r}^{s}\right)}+\|g\|_{L_{T}^{\rho_{3}}\left(L^{\infty}\right)}\|g\|_{\LL_{T}^{\rho_{4}}\left(\dot{B}_{p, r}^{s}\right)}\right),
$$
where
$$
\frac{1}{\rho}=\frac{1}{\rho_{1}}+\frac{1}{\rho_{2}}=\frac{1}{\rho_{3}}+\frac{1}{\rho_{4}}.
$$
\end{lemma}
\begin{lemma}[see \cite{n1}]\label{n1}
Let $1\leq \rho,\rho_1,\rho_2\leq \infty$ with $\frac1\rho=\frac1\rho_1+\frac1\rho_2$ and $1\leq p<2d$. Then, we have
\bbal
\|fg\|_{\LL^\rho_T(\B^{\frac dp-1}_{p,1})}\leq C\|f\|_{\LL^{\rho_1}_T(\B^{\frac dp-1}_{p,1})}\|g\|_{\LL^{\rho_2}_T(\B^{\frac dp}_{p,1})}.
\end{align*}
\end{lemma}

\begin{lemma}[see \cite{n1}]\label{n2}
Let $1\leq \rho,\rho_1,\rho_2\leq \infty$ with $\frac1\rho=\frac1\rho_1+\frac1\rho_2$ and $d< p<2d\leq q < \infty$ with $\frac dp+\frac dq>1$. Then, we have
\bbal
\|fg\|_{\LL^\rho_T(\B^{\frac dp-1}_{p,1})}\leq C\|f\|_{\LL^{\rho_1}_T(\B^{\frac dp-1}_{p,1})}\|g\|_{\LL^{\rho_2}_T(\B^{\frac dq}_{q,1})}.
\end{align*}
\end{lemma}
Finally, we recall the regularity estimates for the heat equations.
\begin{lemma}[see \cite{B}]\label{rg}
Let $s\in \mathbb{R}$, $1\leq p,r\leq \infty$  and $1\leq q_1\leq q_2\leq \infty$. Assume that $u_0\in \dot{B}^s_{p,r}$ and $f\in {\LL}^{q_1}_T(\dot{B}^{s+\frac{2}{q_1}-2}_{p,r})$. Then the heat equations
\begin{align*}
\left\{\begin{array}{ll}
\partial_tu-\Delta u=f,\\
 u(0,x)=u_0(x),
\end{array}\right.
\end{align*}
has a unique solution $u\in \LL^{q_2}_T(\dot{B}^{s+\fr{2}{q_2}}_{p,r}))$ satisfying for all $T>0$
\begin{align*}
\|u\|_{\LL^{q_2}_T(\dot{B}^{s+\fr{2}{q_2}}_{p,r})}\leq C\left(\|u_0\|_{\dot{B}^s_{p,r}}+\|f\|_{{\LL}^{q_1}_T(\dot{B}^{s+\frac{2}{q_1}-2}_{p,r})}\right).
\end{align*}
\end{lemma}

\section{Proof of Theorem \ref{th3}}
\subsection{Construction of initial data}
Letting $n \gg 1,$ we write
$$n\in 16\mathbb{N}=\left\{16,32,48,\cdots\right\}\quad\text{and}\quad\mathbb{N}(n)=\left\{k\in 8\mathbb{N}: \frac{n}4 \leq k\leq \frac{n}2\right\},$$
Before constructing the initial data $(u_{0},v_{0})$, we need to introduce smooth, radial cut-off functions to localize the frequency region.
Let $\widehat{\theta}\in \mathcal{C}^\infty_0(\mathbb{R})$  be an even, real-valued function with values in $[0,1]$ and satisfy
\bbal
\widehat{\theta}(\xi)=
\bca
1, \quad \mathrm{if} \ |\xi|\leq \frac{1}{200d},\\
0, \quad \mathrm{if} \ |\xi|\geq \frac{1}{100d}.
\eca
\end{align*}
Let $\ee=(1,0,\cdots,0)$ and
$$\phi(x_1,x_2,\cdots,x_d)=\theta(x_1)\theta(x_2)\cdots\theta(x_d)
\sin\left(\frac{17}{24}x_d\right).$$
We define
\bal
f_n&\equiv n^{-\frac {1}{2r}}\sum\limits_{k\in \mathbb{N}(n)}2^{\frac k2}\phi\left(2^{k}(x-2^{2n+k}\ee)\right)
\sin\left(\frac{17}{12}2^{n}x_1\right).\label{b}
\end{align}
It is straightforward to verify that
\bal\label{con}
\mathrm{supp} \ \widehat{f_n}(\xi)&\subset  \left\{\xi\in\R^d: \ \frac{33}{24}2^{n}\leq |\xi|\leq \frac{35}{24}2^{n}\right\}.
\end{align}
We construct initial data $(u_{0,n},v_{0,n})$ as follows
\begin{equation}\label{h}
u_{0,n}= 2^{\frac32n}f_n  \quad\text{and}\quad v_{0,n}= 2^{\frac12n}f_n\ee.
\end{equation}

It is worth highlighting again the key step of our proof. As mentioned above, we first consider the strong solution $v$ by the integral form \eqref{v0}. By the linearized equation of $\eqref{che}_1$, we introduce the first approximation $U_1=e^{t\Delta}u_{0,n}$ of $u$. By using $(U_1,v)$, we extract the worst term $U_{2,1}$ of $u$ that primarily affects the discontinuity to the original solution as $n\to\infty$, namely, $U_{2,1}=\int_{0}^{t} e^{(t-s) \Delta}\D (u_{0,n}v_{0,n})\mathrm{d}s$. Naturally, we can construct the initial data $(u_{0,n},v_{0,n})$ by \eqref{h} and hope that $U_{2,1}=2^{2n}\int_{0}^{t} e^{(t-s) \Delta}\pa_{x_1} (f_n^2)\mathrm{d}s$ can lead to the discontinuity of the original solution.
Next we give some explanations to the construction of $f_n$. We can assume that $\widehat{\theta}(\xi)$ is ``good" such that $\theta$ is a real function. Introducing $\phi$ is to guarantee that $\mathcal{F}\big(\phi\big(2^{k}(x-2^{2n+k}\ee)\big)\big)$ is supported in the small dyadic $\{\xi: |\xi|\sim 2^k\}$, it follows that $\widehat{f_n}$ is supported in the big dyadic $\{\xi: |\xi|\sim 2^n\}$ which means that the frequency of $f_n$ is very high if $n$ is very large. The quadratic term $f_n^2$ will generate many high frequency terms (whose fourier transform is supported in different dyadic regions), see \eqref{sub}. To eliminate these high frequency terms as many as possible, the explicit coefficients are introduced. Lastly, we should mention that the translation transform is needed particularly in the proof of Lemma \ref{d2} and \eqref{k2} below.

\subsection{Estimation of initial data}

\begin{lemma}\label{lem1}
Let $f_n$ be defined by \eqref{b}. Then for $(p,r)\in[1, 2d]\times[1,d)$, there exists a positive constant $C$ independent of $n$ such that
\bal\label{d2}
\|f_n\|_{L^{p}}\leq C2^{\frac{p-2d}{8p}n}n^{\frac{1}{p}-\frac{1}{2r}}.
\end{align}
In particular, by $L^{7/2}=[L^{2},L^{4}]_{(1/7,6/7)}$, then \eqref{d2} holds for $d=2$ and $p=7/2$.
\end{lemma}
\begin{proof}
The proof is postponed to {\bf{A.1}} in Section \ref{sec4}.
\end{proof}

As an application of Lemma \ref{lem1}, we have
\begin{proposition}\label{pro1}
Let $(u_{0,n},v_{0,n})$ be defined by \eqref{h}. Then for $\sigma\in \R$ and $(p,r)\in[1, 2d]\times[1,d)$, there exists a positive constant $C$ independent of $n$ such that
\bal\label{u0}
&\|u_{0,n}\|_{\B^{\sigma}_{p,r}}+\|v_{0,n}\|_{\B^{\sigma+1}_{p,r}}\leq C2^{n(\sigma+\frac32)}2^{\frac{p-2d}{8p}n}n^{\frac{1}{p}-\frac{1}{2r}}.
\end{align}
In particular, it holds that
\bbal
\|u_{0,n}\|_{\B^{-\frac32}_{2d,r}}+\|v_{0,n}\|_{\B^{-\frac12}_{2d,r}}\leq Cn^{\frac{1}{2d}-\frac{1}{2r}}.
\end{align*}
\end{proposition}
\begin{proof}
 Notice that $\widehat{\dot{\Delta}_jf_n}=\varphi(2^{-j}\cdot)\widehat{f_n}$ for all $j\in \mathbb{Z}$ and
$
\varphi(2^{-j}\xi)\equiv 1$ in $\left\{\xi\in\R^d: \ \frac{4}{3}2^{j}\leq |\xi|\leq \frac{3}{2}2^{j}\right\},
$
then we have
$
\widehat{\dot{\Delta}_jf_n}=0$ for  $j\neq n,
$
 and thus,
\begin{numcases}
{\dot{\Delta}_j(f_n)=}
f_n, &if $j=n$,\nonumber\\
0, &otherwise.\nonumber
\end{numcases}
Using the definition of Besov space, \eqref{con} and Lemma \ref{lem1} yields \eqref{u0}. This completes the proof of Proposition \ref{pro1}.
\end{proof}
\subsection{Key Estimations}
\quad
From now on, we choose ``certain time" as $t_n=\ep 2^{-2n}$ where $0<\ep\ll1$ will be fixed later and set $T=2^{-2n}$.

{\bf Step 1: Estimation of $U_1$}.

Recalling that $\pa_tU_1-\De U_1=0$ with
$U_1|_{t=0}=u_{0,n}$, using Lemma \ref{rg} and Proposition \ref{pro1}, one has
\bal
\|U_{1}\|_{\tilde{L}^\infty_{T}\B^{-\frac32}_{2d,r}}
\leq C\|u_{0,n}\|_{\B^{-\frac32}_{2d,r}}\leq Cn^{\frac{1}{2d}-\frac{1}{2r}}.
\end{align}
We should notice that, since the initial data $u_{0,n}$ is in the Schwartz class, we can deduce that $U_1$ belongs to the smoother class. More precisely, we know from Lemma \ref{rg} that
\bbal
U_{1}\in \mathcal{C}\Big([0,T];\B^{\frac{d}{p_0}-2}_{p_0,1}\cap\B^{-\frac{3}{2}}_{2d,1}\Big)\cap \LL^1\Big(0,T;\B^{\frac{d}{p_0}}_{p_0,1}\cap\B^{\frac{1}{2}}_{2d,1}\Big),
\end{align*}
where $ 1\leq p_0< 2d$, and $U_{1}$ satisfies the following
\bal\label{u1}
&\|U_{1}\|_{\tilde{L}^\infty_{T}(\B^{\frac{d}{p_0}-2}_{p_0,1})}+\|U_{1}\|_{\LL^1_{T}(\B^{\frac{d}{p_0}}_{p_0,1})}\leq C2^{(\frac{3d}{4p_0}-\fr38)n}n^{\frac{1}{p_0}-\frac{1}{2}},\\
&\|U_{1}\|_{\tilde{L}^\infty_{T}(\B^{-\frac{3}{2}}_{2d,1})}+\|U_{1}\|_{\LL^1_{T}(\B^{\frac{1}{2}}_{2d,1})}\leq Cn^{\frac{1}{2d}-\frac{1}{2}}.
\end{align}
In view of $V_1=v_{0,n}+\int_0^t\nabla U_1\dd\tau$, thus we have
\bbal
V_{1}\in \mathcal{C}\Big([0,T];\B^{\frac{d}{p_0}-1}_{p_0,1}\Big)\cap \LL^2\Big(0,T;\B^{\frac{1}{2}}_{2d,1}\Big).
\end{align*}
In fact, it follows from \eqref{u1} that
\bal\label{v1}
&\|V_1\|_{\LL^\infty_{T}(\B^{\fr{d}{p_0}-1}_{p_0,1})}\leq \|v_{0,n}\|_{\B^{\fr{d}{p_0}-1}_{p_0,1}}+\|U_1\|_{\LL^1_{T}(\B^{\fr{d}{p_0}}_{p_0,1})}\leq C2^{n(\fr{3d}{4p_0}-\frac38)}n^{\frac{1}{p_0}-\frac{1}{2}},\\
&\|V_1\|_{\LL^2_{T}(\B^{\fr{1}{2}}_{2d,1})}\leq T^\fr12\left(\|v_{0,n}\|_{\B^{\fr{1}{2}}_{2d,1}}+\|U_1\|_{\LL^1_{T}(\B^{\frac{3}{2}}_{2d,1})}\right)\leq Cn^{\frac{1}{2d}-\frac{1}{2}}.
\end{align}
{\bf Step 2: Estimation of $U_2$}.

Recalling that $\pa_tU_2-\De U_2=\D(U_1V_1)$ with $U_2|_{t=0}=0$ and $V_2(t)=\int_0^t\nabla U_2\dd s$, from Step 1, one has
\bbal
U_{2}\in  \mathcal{C}\Big([0,T];\B^{\frac{d}{p_0}-2}_{p_0,1}\Big)\cap\LL^1\Big(0,T;\B^{\frac{d}{p_0}}_{p_0,1}\cap\B^{\frac{1}{2}}_{2d,1}\Big)\quad\text{and}\quad
V_{2}\in \mathcal{C}\Big([0,T];\B^{\frac{d}{p_0}-1}_{p_0,1}\Big)\cap \LL^2\Big(0,T;\B^{\frac{1}{2}}_{2d,1}\Big).
\end{align*}
Indeed, from Lemma \ref{rg} and Lemma \ref{n1}, it follows that
\bbal
\|U_{2}\|_{\tilde{L}^\infty_{T}(\B^{\frac{d}{p_0}-2}_{p_0,1})}&\leq\|U_1V_1\|_{\LL^1_{T}(\B^{\fr{d}{p_0}-1}_{p_0,1})}
\leq\|U_1\|_{\LL^1_{T}(\B^{\fr{d}{p_0}}_{p_0,1})}\|V_1\|_{\LL^\infty_{T}(\B^{\fr{d}{p_0}-1}_{p_0,1})}
\leq C2^{(\frac{3d}{2p_0}-\fr34)n}n^{\frac{2}{p_0}-1}
\end{align*}
and
\bal\label{u2}
&\|U_2\|_{\LL^1_{T}(\B^{\fr{d}{p_0}}_{p_0,1})}+\|V_2\|_{\LL^\infty_{T}(\B^{\fr{d}{p_0}-1}_{p_0,1})}\nonumber\\
\leq&\ C T^{\fr12}\|U_2\|_{\LL^2_{T}(\B^{\fr{d}{p_0}}_{p_0,1})}
\leq C T^{\fr12}\|U_1V_1\|_{\LL^1_{T}(\B^{\fr{d}{p_0}}_{p_0,1})}\nonumber\\
\leq&\ CT^{\fr12}\left(\|U_1\|_{L^\infty_{T}(L^\infty)}\|V_1\|_{\LL^1_{T}(\B^{\fr{d}{p_0}}_{p_0,1})}
+\|V_1\|_{L^\infty_{T}(L^\infty)}\|U_1\|_{\LL^1_{T}(\B^{\fr{d}{p_0}}_{p_0,1})}\right)\nonumber\\
\leq&\ C2^{n(\fr{3d}{4p_0}-\frac{5}{8})}n^{\frac{1}{p_0}-\frac{1}{2r}-\frac{1}{2}},
\end{align}
where we have used the following estimates from  the classical $L^\infty$-$L^\infty$ estimate: $\|e^{t\Delta}f\|_{L^\infty} \leq \|f\|_{L^\infty}$ for $t>0$
\bbal
&\|U_1\|_{L^\infty_{T}(L^\infty)}\leq \|u_{0,n}\|_{L^\infty}\leq C2^{\frac{3}{2}n} 2^{\frac{1}{4}n}n^{-\frac{1}{2r}} \leq C2^{\frac{7}{4}n}n^{-\frac{1}{2r}},\\
&\|V_1\|_{L^\infty_{T}(L^\infty)}\leq \|v_{0,n}\|_{L^\infty}+T\|\nabla U_1\|_{L^\infty} \leq C\left(2^{\frac{1}{2}n} 2^{\frac{1}{4}n}+2^{-2n}2^n2^{\frac{3}{2}n} 2^{\frac{1}{4}n}\right)n^{-\frac{1}{2r}}\leq C2^{\frac{3}{4}n}n^{-\frac{1}{2r}},\\
&\|V_1\|_{\LL^1_{T}(\B^{\fr{d}{p_0}}_{p_0,1})}\leq CT\left(\|v_{0,n}\|_{\B^{\fr{d}{p_0}}_{p_0,1}}+T\|u_{0,n}\|_{\B^{\fr{d}{p_0}+1}_{p_0,1}}\right)\leq C2^{n(\fr{d}{p_0}-\frac{3}{2})}2^{\frac{p_0-2d}{8p_0}n}n^{\frac{1}{p_0}-\frac{1}{2}}.
\end{align*}
Similarly, for $q_0=2d$
\bal\label{v2}
\|U_2\|_{\LL^1_{T}(\B^{\fr{d}{q_0}}_{q_0,1})}+\|V_2\|_{\LL^2_{T}(\B^{\fr{d}{q_0}}_{q_0,1})}
\leq C2^{-\frac{n}{4}}n^{\frac{1}{q_0}-\frac{1}{2r}-\frac{1}{2}}.
\end{align}
{\bf Step 3: Estimation of $U_3$}.

We choose the index $(d,p_0,q_0)$ to satisfy that $q_0=2d$ and
\bbal
p_0=\bca
\fr72, \quad  d=2,\\
2d-1, \quad  d\geq 3.
\eca
\end{align*}
Obviously, it holds that
\bal\label{cond}
1\leq r<d<p_0<2d=q_0\quad \text{and}\quad\frac{3d}{2p_0}-1<0.
\end{align}
 For the sake of convenience, for $T= 2^{-2n}$, we denote
\bbal
&X_T=\|U_3(t,\cdot)\|_{\LL^\infty_T(\B^{\frac{d}{p_0}-2}_{p_0,1})}
+\|U_3(t,\cdot)\|_{\LL^1_T(\B^{\frac{d}{p_0}}_{p_0,1})}\quad\text{and}\quad
Y_T=\|V_3(t,\cdot)\|_{\LL^\infty_T(\B^{\frac{d}{p_0}-1}_{p_0,1})}.
\end{align*}
Obviously,
$$Y_T\leq C \|U_3(t,\cdot)\|_{\LL^1_T(\B^{\frac{d}{p_0}}_{p_0,1})}\leq CX_T.$$
Utilizing Lemma \ref{rg} to \eqref{3}, we have
\bal\label{l1}
X_T&\leq C\|U_3V_3+U_3(V_1+V_2)+V_3(U_1+U_2)+U_1V_2+U_2(V_1+V_2)\|_{\LL^1_T(\B^{\frac{d}{p_0}-1}_{p_0,1})}.
\end{align}
Utilizing Lemma \ref{n1}, one has
\bal
&\|U_3V_3\|_{\LL^1_T(\B^{\frac{d}{p_0}-1}_{p_0,1})}
\leq C\|U_3\|_{\LL^1_T(\B^{\frac{d}{p_0}}_{p_0,1})}\|V_3\|_{\LL^\infty_T(\B^{\frac{d}{p_0}-1}_{p_0,1})}\leq CX_T^2,\label{j1}\\
&\|U_2(V_1+V_2)\|_{\LL^1_T(\B^{\frac{d}{p_0}-1}_{p_0,1})}
\leq C\|U_2\|_{\LL^1_T(\B^{\frac{d}{p_0}}_{p_0,1})}\|V_1,V_2\|_{\LL^\infty_T(\B^{\frac{d}{p_0}-1}_{p_0,1})},\label{j2}\\
&\|U_1V_2\|_{\LL^1_T(\B^{\frac{d}{p_0}-1}_{p_0,1})}
\leq C\|U_1\|_{\LL^1_T(\B^{\frac{d}{p_0}}_{p_0,1})}\|V_2\|_{\LL^\infty_T(\B^{\frac{d}{p_0}-1}_{p_0,1})}.\label{j3}
\end{align}
Utilizing Lemma \ref{n2}, one has
\bal\label{j4}
\|U_3(V_1+V_2)\|_{\LL^1_T(\B^{\frac{d}{p_0}-1}_{p_0,1})}
&\leq C\|U_3\|_{\LL^2_T(\B^{\frac{d}{p_0}-1}_{p_0,1})}\|V_1+V_2\|_{\LL^2_T(\B^{\frac{d}{q_0}}_{q_0,1})}\nonumber\\
&\leq C\|U_3\|^{\fr12}_{\LL^\infty_T(\B^{\frac{d}{p_0}-2}_{p_0,1})}\|U_3\|^{\fr12}_{\LL^1_T(\B^{\frac{d}{p_0}}_{p_0,1})}
\|V_1,V_2\|_{\LL^2_T(\B^{\frac{d}{q_0}}_{q_0,1})}\nonumber\\
&\leq CX_T
\|V_1,V_2\|_{\LL^2_T(\B^{\frac{d}{q_0}}_{q_0,1})}
\end{align}
and
\bal\label{j5}
\|V_3(U_1+U_2)\|_{\LL^1_T(\B^{\frac{d}{p_0}-1}_{p_0,1})}
&\leq C\|V_3\|_{\LL^\infty_T(\B^{\frac{d}{p_0}-1}_{p_0,1})}\|U_1+U_2\|_{\LL^1_T(\B^{\frac{d}{q_0}}_{q_0,1})}
\leq CX_T\|U_1,U_2\|_{\LL^1_T(\B^{\frac{d}{q_0}}_{q_0,1})}.
\end{align}
Inserting \eqref{j1}-\eqref{j5} into \eqref{l1} yields
\bal\label{j6}
X_T&\leq CX^2_T+CX_T\Big(\|U_1,U_2\|_{\LL^1_T(\B^{\frac{d}{q_0}}_{q_0,1})}+\|V_1,V_2\|_{\LL^2_T(\B^{\frac{d}{q_0}}_{q_0,1})}\Big)
\nonumber\\&\quad + C\|U_2\|_{\LL^1_T(\B^{\frac{d}{p_0}}_{p_0,1})}
\|V_1,V_2\|_{\LL^\infty_T(\B^{\frac{d}{p_0}-1}_{p_0,1})}+
C\|U_1\|_{\LL^1_T(\B^{\frac{d}{p_0}}_{p_0,1})}
\|V_2\|_{\LL^\infty_T(\B^{\frac{d}{p_0}-1}_{p_0,1})}.
\end{align}
Due to \eqref{u1}-\eqref{v2}, we have for $q_0=2d$
\bbal
&C X_T\Big(\|U_1,U_2\|_{\LL^1_T(\B^{\frac{d}{q_0}}_{q_0,1})}+\|V_1,V_2\|_{\LL^2_T(\B^{\frac{d}{q_0}}_{q_0,1})}\Big)\leq Cn^{\frac{1}{2d}-\frac{1}{2}}X_T,
\\& C\|U_2\|_{\LL^1_T(\B^{\frac{d}{p_0}}_{p_0,1})}
\|V_1,V_2\|_{\LL^\infty_T(\B^{\frac{d}{p_0}-1}_{p_0,1})}+
C\|U_1\|_{\LL^1_T(\B^{\frac{d}{p_0}}_{p_0,1})}
\|V_2\|_{\LL^\infty_T(\B^{\frac{d}{p_0}-1}_{p_0,1})}\leq C2^{n(\frac{3d}{2p_0}-1)}n^{\frac{2}{p_0}}.
\end{align*}
Putting the above inequalities together with \eqref{j6} yields
\bbal
X_T&\leq CX^2_T+Cn^{\frac{1}{2d}-\frac{1}{2}}X_T+
C2^{n(\frac{3d}{2p_0}-1)}n^{\frac{2}{p_0}}.
\end{align*}
By using the continuity argument and condition \eqref{cond}, we can take $n$ large enough such that
\bbal
X_T\leq C2^{n(\frac{3d}{2p_0}-1)}n^{\frac{2}{p_0}}.
\end{align*}
Then, by the embedding $\B^{\frac{d}{p_0}-2}_{p_0,1}(\R^d)\hookrightarrow\B^{-\frac32}_{2d,r}(\R^d)$, we have
\bbal
\|U_3(t_n,\cdot)\|_{\B^{-\frac32}_{2d,r}}\leq CX_T\leq C2^{n(\frac{3d}{2p_0}-1)}n^{\frac{2}{p_0}}.
\end{align*}
\subsection{Ill-posedness}
\quad
Now, we can decompose $U_2$ as follows
$$U_2=U_{2,1}+U_{2,2},$$ where $U_{2,1}$ and $U_{2,2}$ satisfy respectively
\begin{equation}\label{z2}
\begin{cases}
\pa_tU_{2,1}-\De U_{2,1}=\D(u_{0,n}v_{0,n})=2^{2n}\pa_{x_1}(f_n^2),\\
U_{2,1}|_{t=0}=0,
\end{cases}
\end{equation}
and
\begin{equation}\label{z1}
\begin{cases}
\pa_tU_{2,2}-\De U_{2,2}=\D\left(U_1\int_0^t\nabla U_1\dd\tau+(U_1-u_{0,n})v_{0,n}\right),\\
U_{2,2}|_{t=0}=0.
\end{cases}
\end{equation}
Following the above argument, we know that for $i=1,2$
\bbal
U_{2,i}\in  \mathcal{C}\Big([0,T];\B^{\frac{d}{p_0}-2}_{p_0,1}\Big)\cap\LL^1\Big(0,T;\B^{\frac{d}{p_0}}_{p_0,1}\cap\B^{\frac{1}{2}}_{2d,1}\Big).
\end{align*}
 Using Lemma \ref{rg} and noticing that $U_1(t)-u_{0,n}=\int_0^t\Delta U_1\dd\tau$ yields for $t_n=\ep2^{-2n}$
\bal\label{lll}
\|U_{2,2}\|_{L^\infty_{t_n}(\B^{-\frac32}_{2d,r}(\mathbb{N}(n)))}&\leq\|U_{2,2}\|_{L^\infty_{t_n}(\B^{-1}_{d,r}(\mathbb{N}(n)))}\leq  Cn^{\fr{1}r-\fr{1}d}\|U_{2,2}\|_{L^\infty_{t_n}(\B^{-1}_{d,d}(\mathbb{N}(n)))}\nonumber\\
&\leq  Ct_nn^{\fr{1}r-\fr{1}d}\left(\left\|U_1\cdot\int_0^{t_n}\nabla U_1\dd\tau\right\|_{L^\infty_{t_n}(\B^{0}_{d,d})}+\left\|\int_0^{t_n}\Delta U_1\dd\tau \cdot v_{0,n}\right\|_{L^\infty_{t_n}(\B^{0}_{d,d})}\right)\nonumber\\
&\leq  Ct_n^2n^{\fr{1}r-\fr{1}d} \left(\|U_{1}\|_{L^\infty_{t_n}(L^{2d})}
\|\nabla U_{1}\|_{L^\infty_{t_n}(L^{2d})}+\|\Delta U_1\|_{L^\infty_{t_n}(L^{2d})}
\|v_{0,n}\|_{L^\infty_T(L^{2d})}\right)\nonumber\\
&\leq Ct_n^2n^{\fr{1}r-\fr{1}d} \left(2^n\|u_{0,n}\|_{L^{2d}}^2+2^{2n}\|u_{0,n}\|_{L^{2d}}
\|v_{0,n}\|_{L^{2d}}\right)\nonumber\\
&\leq C\ep^2.
\end{align}
Next, we give the lower bound estimation of $\|U_{2,1}(t_n,\cdot)\|_{\B^{-\frac32}_{2d,r}(\mathbb{N}(n))}$ which is crucial for the proof of the discontinuity of solutions.

Taking advantage of the Duhamel formula, then we have from \eqref{z2}
\bbal
U_{2,1}(t_n,\cdot)&=2^{2n}\int^{t_n}_0e^{({t_n}-\tau)\Delta}\pa_{x_1}(f^2_n)\dd \tau.
\end{align*}
Direct computations gives that for $\ell\in \mathbb{N}(n)$
\bbal
\dot{\Delta}_{\ell}U_{2,1}(t_n,\cdot)&=2^{2n}\int^{t_n}_0\mathcal{F}^{-1}\left(\varphi_{\ell}(\xi) e^{-(t_n-\tau)|\xi|^2}\mathcal{F}(\pa_{x_1}(f^2_n))\right)\dd \tau\\
&=2^{2n}\mathcal{F}^{-1}\left(\varphi_{\ell}(\xi) \frac{1-e^{-t_{n}|\xi|^{2}}}{|\xi|^{2}} \mathcal{F}(\pa_{x_1}(f^2_n)) \right)\\
&=\ep\left(\dot{\Delta}_{\ell}(\pa_{x_1}(f^2_n))
+\sum_{k\geq1}\frac{t_n^{k}}{(k+1)!}\dot{\Delta}_{\ell}(\pa_{x_1}\Delta^k(f^2_n))\right),
\end{align*}
where we have used Taylor's formula
$$\frac{1-e^{-t_{n}|\xi|^{2}}}{|\xi|^{2}}=t_n+t_n\sum_{k\geq1}\frac{t_n^{k}}{(k+1)!}(-|\xi|^2)^{k}.$$
By Bernstein's inequality, we deduce that
 \bbal
\left\|\sum_{k\geq1}\frac{t_n^{k}}{(k+1)!}\dot{\Delta}_{\ell}(\pa_{x_1}\Delta^k(f^2_n))\right\|_{L^{2d}}&\leq C\sum_{k\geq1}\frac{t_n^{k}}{(k+1)!}\left\|\dot{\Delta}_{\ell}(\pa_{x_1}\Delta^k(f^2_n))\right\|_{L^{2d}}\\
&\leq C\sum_{k\geq1}\frac{\ep^{k}}{(k+1)!}\left\|\dot{\Delta}_{\ell}\pa_{x_1}(f_n^2)\right\|_{L^{2d}}\\
&\leq C\ep\left\|\dot{\Delta}_{\ell}\pa_{x_1}(f_n^2)\right\|_{L^{2d}}.
\end{align*}
By the inverse triangle inequality, we have
\bal\label{y}
\|U_{2,1}(t_n,\cdot)\|_{\B^{-\frac32}_{2d,r}(\mathbb{N}(n))}&\geq \ep(1-C\ep) \left(\sum_{\ell\in\mathbb{N}(n)}2^{-\fr32r\ell}\left\|\dot{\Delta}_{\ell}\pa_{x_1}(f^2_{n})\right\|^r_{L^{2 d}\left(\mathbb{R}^{d}\right)}\right)^{\fr1r}.
\end{align}
For $k\in \mathbb{N}(n)$, using the simple fact $\sin^2\alpha=(1-\cos2\alpha)/2$, we decompose the term $n^{\frac {1}{r}}f_n^2$ as
\bal\label{sub}
n^{\frac {1}{r}}f_n^2&=\fr12 \sum\limits_{k\in \mathbb{N}(n)}2^{k}\phi^2\left(2^{k}(x-2^{2n+\ell}\ee)\right)
+\mathbf{G}+\mathbf{H},
\end{align}
where
\bbal
&\mathbf{G}:=\fr12\sum\limits_{k\in \mathbb{N}(n)}2^{k}\phi^2\left(2^{k}(x-2^{2n+\ell}\ee)\right)
\cos\left(\frac{17}{12}2^{n+1}x_1\right),\nonumber\\
&\mathbf{H}:=\fr12\sum\limits_{k,j\in \mathbb{N}(n)\atop
k\neq j}2^{\frac{k}2}2^{\frac{j}2}\phi\left(2^{k}(x-2^{2n+k}\ee)\right)\phi\left(2^{j}(x-2^{2n+\ell}\ee)\right)\left(
1-\cos\left(\frac{17}{12}2^{n+1}x_1\right)\right).
\end{align*}
Noticing that (for more details see {\bf{A.2}} in Appendix)
\bal\label{ob}
\dot{\Delta}_{\ell}\mathbf{G}=\dot{\Delta}_{\ell}\mathbf{H}=0\quad\text{for}\quad\ell\in \mathbb{N}(n),
\end{align}
then we have
\bbal
n^{\frac {1}{r}}\dot{\Delta}_{\ell}f_n^2&=
\fr12\dot{\Delta}_{\ell}\left(2^{\ell}\phi^2\big(2^{\ell}(x-2^{2n+\ell}\ee)\big)\right)+\fr12\dot{\Delta}_{\ell}\left(\sum\limits_{k\in \mathbb{N}(n)\atop
k\neq \ell}2^{k}\phi^2\big(2^{k}(x-2^{2n+k}\ee)\big)\right),
\end{align*}
which in turn gives
\bbal
n^{\frac {1}{r}}\pa_{x_1}\dot{\Delta}_{\ell}f_n^2&=
2^{2\ell-1}[h]\big(2^{\ell}(x-2^{2n+\ell}\ee)\big)+\fr12\pa_{x_1}\dot{\Delta}_{\ell}\left(\sum\limits_{k\in \mathbb{N}(n)\atop
k\neq \ell}2^{k}\phi^2\big(2^{k}(x-2^{2n+k}\ee)\big)\right)\\
&=:K_1+K_2,
\end{align*}
where
\bbal
h(x)
&=-\theta(x_1)\theta'(x_1)\theta^2(x_2)\theta^2(x_3)\cdots\theta^2(x_{d-1})\theta^2(x_d)
\cos\left(\frac{17}{12}x_d\right).
\end{align*}
Defining the set $\mathbf{B}_{\ell}$ by
$$
\mathbf{B}_{\ell}\equiv\left\{x:\big|2^{\ell}(x-2^{2n+\ell} \ee)\big|\leq 1\right\},
$$
then by change of variables, we have
\bbal
\|K_1\|_{L^{2d}(\mathbf{B}_{\ell})}
= 2^{2\ell-1}\left\|[h]\big(2^{\ell}(x-2^{2n+\ell}\ee)\big)\right\|_{L^{2d}(\mathbf{B}_{\ell})}
= 2^{\fr32\ell-1}\left\|h(y)\right\|_{L^{2d}(|y|\leq 1)}
= \tilde{c}2^{\fr32\ell}.
\end{align*}
Combining the estimate whose proof is relegated to {\bf{A.3}} in Section \ref{sec4}
\bal\label{k2}
\|K_2\|_{L^{2d}(\mathbf{B}_{\ell})}\leq C2^{-n},
\end{align}
thus for $\ell\in \mathbb{N}(n)$, we have
\bal\label{z}
\left\|\dot{\Delta}_{\ell}\pa_{x_1}(f^2_{n})\right\|_{L^{2 d}\left(\mathbf{B}_{\ell}\right)}&\geq n^{-\fr1r}\left(\left\|K_1\right\|_{L^{2 d}\left(\mathbf{B}_{\ell}\right)}-\left\|K_2\right\|_{L^{2 d}\left(\mathbf{B}_{\ell}\right)}\right)\nonumber\\
&\geq n^{-\fr1r}\left(\tilde{c}2^{\fr32\ell}-C2^{-n}\right)\nonumber\\
&\geq cn^{-\fr1r}2^{\fr32\ell}.
\end{align}
Inserting \eqref{z} into \eqref{y}, and from \eqref{lll} we conclude that for $\ep$ small enough
\bal\label{hhh}
\|U_{2}(t_n,\cdot)\|_{\B^{-\frac32}_{2d,r}(\mathbb{N}(n))}
&\geq \ep(c-C\ep)\geq C\ep.
\end{align}

Combining \eqref{hhh} and {\bf Step1}-{\bf Step3}, we obtain that for large $n$ enough and $\ep$ small enough
\bbal
\|u_{0,n}\|_{\B^{-\frac32}_{2d,r}}+\|v_{0,n}\|_{\B^{-\frac12}_{2d,r}}\leq Cn^{\frac{1}{2d}-\frac{1}{2r}}\rightarrow0,\;n\rightarrow\infty
\end{align*}
and
\bbal
\|u(t_n)\|_{\B^{-\frac32}_{2d,r}}&\geq \|U_2(t_n)\|_{\B^{-\frac32}_{2d,r}(\mathbb{N}(n))}
-\|U_1(t_n)\|_{\B^{-\frac32}_{2d,r}}
-\|U_3(t_n)\|_{\B^{-\frac32}_{2d,r}}
\\&\geq C\ep-Cn^{\frac{1}{2d}-\frac{1}{2r}}-C2^{n(\frac{3d}{2p_0}-1)}n^{\frac{2}{p_0}}\\&\geq \fr{C}2\ep.
\end{align*}
Thus, we have obtained a sequence of initial data such that it verifies the discontinuity of data-to-solution map. The proof of Theorem \ref{th3} is finished.{\hfill $\square$}

\section{Appendix}\label{sec4}
\setcounter{equation}{0}
For the sake of convenience, here we present more details in the computations.\\

{\bf A.1\quad Proof of Lemma \ref{lem1}.}\\
We assume that $p\in \mathbb{Z}^+$ without loss of generality. Since $\phi$ is a Schwartz function, we have for $100d\leq M\in \mathbb{Z}^+$
\bbal
|\phi(x)|\leq C(1+|x|)^{-M}.
\end{align*}
It is easy to show that
\bal\label{a0}
\left\|f_n\right\|^p_{L^p}&\leq n^{-\frac{p}{2r}} \int_{\R^d}\sum\limits_{\ell_1,\ell_2,\cdots,\ell_p\in \mathbb{N}(n)}\frac{2^{\frac12(\ell_1+\ell_2+\cdots+\ell_p)}}{(1+2^{\ell_1}|x
-2^{2n+\ell_1}\ee|)^M\cdots (1+2^{\ell_p}|x-2^{2n+\ell_p}\ee|)^M}\dd x
\nonumber\\&\leq n^{-\frac{p}{2r}}\sum_{\ell\in \mathbb{N}(n)}\int_{\R^d}\frac{2^{\frac{p}{2}\ell}}{(1+2^{\ell}|x-2^{2n+\ell}\ee|)^{pM}}\dd x\nonumber\\
&\quad+n^{-\frac{p}{2r}}\sum\limits_{(\ell_1,\ell_2,\cdots,\ell_p)\in \Lambda}\int_{\R^d}\frac{2^{\frac12(\ell_1+\ell_2+\cdots+\ell_p)}}{(1+2^{\ell_1}|x
-2^{2n+\ell_1}\ee|)^M\cdots (1+2^{\ell_p}|x-2^{2n+\ell_p}\ee|)^M}\dd x
\nonumber\\&\equiv n^{-\frac{p}{2r}}I_1+n^{-\frac{p}{2r}}I_2,
\end{align}
where the set $\Lambda$ is defined by
$$
\Lambda=\left\{(\ell_{1}, \ldots, \ell_{p}) \in \mathbb{N}^p(n)  \mid \exists 1 \leq k, j\leq p \text { s.t. } \ell_{k} \neq \ell_{j} \right\} .
$$
For the term $I_1$, by direct computations, one has
\bal\label{es-I1}
I_1=\sum_{\ell\in \mathbb{N}(n)}2^{(\frac{p}{2}-d)\ell}\int_{\R^d}\frac{1}{(1+|x|)^{pM}}\dd x\leq Cn2^{\frac{p-2d}{8}n}.
\end{align}
For the term $I_2$, we assume that $\ell_1< \ell_2$ without loss of generality, then $\ell_2- \ell_1\geq8$.
\bbal
&\quad\int_{\R^d}\frac{1}{(1+2^{\ell_1}|x-2^{2n+\ell_1}\ee|)^M (1+2^{\ell_2}|x-2^{2n+\ell_2}\ee|)^M}\dd x\\
&=\int_{\mathbf{A}_{\ell_1}}\frac{1}{(1+2^{\ell_1}|x
-2^{2n+\ell_1}\ee|)^M (1+2^{\ell_2}|x-2^{2n+\ell_2}\ee|)^M}\dd x\\
&\quad+\int_{\mathbf{A}_{\ell_1}^c}\frac{1}{(1+2^{\ell_1}|x
-2^{2n+\ell_1}\ee|)^M (1+2^{\ell_2}|x-2^{2n+\ell_2}\ee|)^M}\dd x\\
&=I_{2,1}+I_{2,2},
\end{align*}
where we defined the set $A_{\ell_1}$ by
$$
\mathbf{A}_{\ell_1}:=\left\{x:\left|x-2^{2 n+\ell_1} \ee\right| \leq 2^{2n}\right\}.
$$
Thus we obtain
\bal\label{a1}
I_{2,2}
\leq C(2^{\ell_{1}} 2^{2 n})^{-M}\int_{\mathbf{A}_{\ell_1}^c}\frac{1}{(1+2^{\ell_2}|x-2^{2n+\ell_2}\ee|)^M}\dd x
\leq C(2^{\ell_{1}} 2^{2 n})^{-M}2^{-d\ell_2}.
\end{align}
It is easy to deduce that for $x\in \mathbf{A}_{\ell_1}$
\begin{align*}
2^{\ell_{2}}\left|x-2^{2 n+\ell_2} \ee\right| & \geq 2^{\ell_{2}}\left|2^{2 n+\ell_2}-2^{2 n+\ell_1}\ee\right|- 2^{\ell_{2}} 2^{2 n} \geq  2^{\ell_{2}} 2^{2 n}.
\end{align*}
Similarly, we have
\bal\label{a2}
I_{2,1}
\leq (2^{\ell_{2}} 2^{2 n})^{-M}\int_{\mathbf{A}_{\ell_1}}\frac{1}{(1+2^{\ell_1}|x
-2^{2n+\ell_1}\ee|)^M }\dd x
\leq C(2^{\ell_{2}} 2^{2 n})^{-M}2^{-d\ell_1}.
\end{align}
We infer from \eqref{a1} and \eqref{a2} that
\bal\label{es-I2}
&I_2\leq C 2^{-2 Mn}\sum\limits_{(\ell_1,\ell_2,\cdots,\ell_p)\in \Lambda}(2^{-M\ell_{1}}2^{-d\ell_2}+2^{-M\ell_{2}}2^{-d\ell_1})2^{\frac12(\ell_1+\ell_2+\cdots+\ell_p)}\leq C2^{-Mn}.
\end{align}
Inserting \eqref{es-I1} and \eqref{es-I2} into \eqref{a0}, we have for large enough $n$
\bbal
&\|f_n\|_{L^p}\leq C n^{\fr1p-\fr1{2r}}2^{\frac{p-2d}{8p}n}.
\end{align*}
This completes the proof of Lemma \ref{lem1}.{\hfill $\square$}

{\bf A.2\quad Proof of \eqref{ob}.}\\
Notice that
\bbal
\mathbf{H}&=\fr12\sum\limits_{k, j\in \mathbb{N}(n)\atop
k\neq j}2^{\frac{k+j}2}\Phi_{k,j}(x)
\left(1-\cos\left(\frac{17}{12}2^{n+1}x_1\right)\right)
\end{align*}
with
$$\Phi_{k,j}(x):=\phi\big(2^{k}(x-2^{2n+k}\ee)\big)
\phi\big(2^{j}(x-2^{2n+j}\ee)\big),$$
and the definition of $\phi$, we deduce that for $j<k$
\bbal
&\mathrm{supp} \ \widehat{\Phi_{k,j}}\subset   \left\{\xi\in\R^d: \ \frac{33}{48}2^{k}\leq |\xi|\leq \frac{35}{48}2^{k}\right\},\\
\Rightarrow\quad&\mathrm{supp} \ \mathcal{F}\left[\Phi_{k,j}\cos\left(\frac{17}{12}2^{n+1} x_1\right)\right]\subset   \left\{\xi\in\R^d: \ \frac{33}{24}2^{n+1}\leq |\xi|\leq \frac{35}{24}2^{n+1}\right\}.
\end{align*}
Then, for $j<k$, we obtain that $\dot{\Delta}_{\ell}\mathbf{H}=0$, which also holds for $j>k$. Similarly, it holds that $\dot{\Delta}_{\ell}\mathbf{G}=0$.

Thus, we finish the proof of \eqref{ob}.{\hfill $\square$}

{\bf A.3\quad Proof of \eqref{k2}.}\\
Noting the fact that for $100d\leq N\in \mathbb{Z}^+$
$$|\check{\varphi}(x)|\leq C(1+|x|)^{-N},$$
then we have
\bal\label{j-3}
\left\|K_2\right\|_{L^{2d}(\mathbf{B}_{\ell})}
&\leq\sum\limits_{k\in \mathbb{N}(n)\atop k\neq \ell}2^{k}2^{d\ell}2^{\ell}\left\|\int_{\R^d}\check{\varphi}(2^{\ell}(x-y))
\phi^2\big(2^{k}(y-2^{2n+k}\ee)\big)
\dd y\right\|_{L^{2d}(\mathbf{B}_{\ell})}\nonumber\\
&\leq  \sum\limits_{k\in \mathbb{N}(n)\atop k\neq \ell}2^{k}2^{d\ell}2^{\ell}\left\|\int_{\R^d}\Big(1+2^{\ell}|x-y|\Big)^{-N}\Big(1+2^{k}|y-2^{2n+k}\ee|\Big)^{-2N}\dd y\right\|_{L^{2d}(\mathbf{B}_{\ell})}.
\end{align}
Dividing the integral region in terms of $y$ into the following two parts to estimate:
\begin{align*}
\R^d&=\left\{y :\; |y-2^{\ell+2 n} \ee | \leq 2^{2 n}\right\}\cup\left\{y :\; |y-2^{\ell+2 n} \ee|\geq 2^{2 n}\right\}= \mathbf{A}_{1} \cup \mathbf{A}_{2}.
\end{align*}
For $x \in \mathbf{B}_{\ell}$ and $y \in \mathbf{A}_{1}$, we conclude that
$$
\begin{aligned}
\left|y-2^{k+2 n} \ee\right| &=\left|(y-2^{\ell+2 n} \ee)+(2^{\ell+2 n} \ee-2^{k+2 n} \ee)\right|\geq\left|2^{\ell+2 n} -2^{k+2 n}\ee\right|-\left|y-2^{\ell+2 n} \ee\right| \geq 2^{2 n}.
\end{aligned}
$$
For $x \in \mathbf{B}_{\ell}$ and $y \in \mathbf{A}_{2}$, it is easy to check that
$$
\begin{aligned}
|x-y| &\geq\left|y-2^{\ell+2 n} \ee\right|-\left|x-2^{\ell+2 n} \ee\right| \geq2^{2 n}- 2^{-\ell} \geq 2^{2 n-1}.
\end{aligned}
$$
Then, we have
\bbal
&\left\|\int_{\R^d}\Big(1+2^{\ell}|x-y|\Big)^{-N}\Big(1+2^{k}|y-2^{2n+k}\ee|\Big)^{-2N}\dd y\right\|_{L^{2d}(\mathbf{B}_{\ell})}\\
\leq&~  C2^{-2(k+2n)N}\left\|\int_{\mathbf{A}_1}\Big(1+2^{\ell}|x-y|\Big)^{-N}\dd y\right\|_{L^{2d}(\mathbf{B}_{\ell})}\\
&+C2^{-(\ell+2 n)N}\left\|\int_{\mathbf{A}_2}\Big(1+2^{k}|y-2^{2n+k}\ee|\Big)^{-2N}\dd y\right\|_{L^{2d}(\mathbf{B}_{\ell})}\\
\leq&~  C\left(2^{-2d\ell}2^{-2(k+2n)N}+2^{-(\ell+2 n)N}2^{-2dk}\right)2^{-\frac{\ell}{2d}}.
\end{align*}
Plugging the above into \eqref{j-3} yields
\bbal
\left\|K_{2}\right\|_{L^{2d}(\mathbf{B}_{\ell})}\leq C\sum\limits_{k\in \mathbb{N}(n)\atop k\neq \ell}2^{k+\ell}2^{d\ell}\left(2^{-2d\ell}2^{-2(k+2n)N}+2^{-(\ell+2 n)N}2^{-2dk}\right)2^{-\frac{\ell}{2d}}\leq C2^{-n},
\end{align*}
which is nothing but \eqref{k2}.{\hfill $\square$}
\section*{Acknowledgments}
The authors would like to express their gratitude to the anonymous referees for valuable suggestions and comments which greatly improved the paper. J. Li is supported by the National Natural Science Foundation of China (11801090 and 12161004) and Jiangxi Provincial Natural Science Foundation (20212BAB211004). Y. Yu is supported by the National Natural Science Foundation of China (12101011). W. Zhu is supported by the National Natural Science Foundation of China (12201118) and Guangdong
Basic and Applied Basic Research Foundation (2021A1515111018).

\section*{Conflict of interest}
The authors declare that they have no conflict of interest.

\end{document}